\documentclass[reqno]{amsart}
\usepackage{multicol,a4wide,enumerate}
\usepackage{amsmath}
 \usepackage{enumitem}
\usepackage{amsaddr}
\usepackage{color}
\usepackage{cmap,mathtools}
\usepackage{cite}
\usepackage{xcolor}
\usepackage{hyperref}
\hypersetup{
    colorlinks=true,
    linkcolor=blue,
    filecolor=magenta,      
    urlcolor=green,
    citecolor=forestgreen,
}
\usepackage[hyperpageref]{backref}
\definecolor{mediumcarmine}{rgb}{0.69, 0.25, 0.21}
\definecolor{forestgreen}{rgb}{0.13, 0.85, 0.15}
\usepackage[hyperpageref]{backref}
\usepackage{yhmath,amsmath,amsfonts,amssymb, enumerate,amsthm, appendix, caption,lipsum,}
 \usepackage[T1]{fontenc}
 \usepackage{enumitem}
 \usepackage{mathtools, bm}
 \usepackage{cmap}
 \usepackage[hyperpageref]{backref}
 \usepackage{bigints}
 \usepackage{esint}
 \usepackage[margin=0.55in]{geometry}
\newtheorem{definition}{Definition}[section]
\newtheorem{theorem}[definition]{Theorem}
\newtheorem{example}[definition]{Example}

\newtheorem{remark}[definition]{Remark}
\newtheorem{proposition}[definition]{Proposition}
\newtheorem{lemma}[definition]{Lemma}

\numberwithin{equation}{section}
\DeclarePairedDelimiter\abs{\lvert}{\rvert}
\DeclarePairedDelimiter\norm{\lVert}{\rVert}
\makeatletter
\let\oldnorm\norm
\def\norm{\@ifstar{\oldnorm}{\oldnorm*}}

\makeatletter

\newcommand{\noi} {\noindent}

\newcommand{\ra} {\rightarrow}

\DeclareMathAlphabet{\mathpzc}{T1}{pzc}{m}{it}

\def\dr{{\rm d}r}
\def\H{{\mathbb{H}}}

\def\Dh{{{\mathcal D}^{1,2}(\H^N)}}

\def\R{{\mathbb R}}
\def\N{{\mathbb N}}
\def\C{{\mathcal C}}

\def\d{{\rm d}}
\def\dxi{{\rm d}\xi}
\def\deta{{\rm d}\eta}

\def\l2{\mathcal M\log L}

\def\c1Loc{{\C_{loc}^1}}

\def\Gag2{\iint_{\R^{2N}} \frac{(u(x)-u(y))^2}{|x-y|^{N+sp}}\ \dxy}

\def\Gagn2{\iint_{\R^{2N}} \frac{(u_n(x)-u_n(y))^2}{|x-y|^{N+sp}}\ \dxy}

\newcommand{\dxy}{{\mathrm{d}x\mathrm{d}y}}

\makeatletter
\newcommand{\leqnomode}{\tagsleft@true\let\veqno\@@leqno}
\makeatother
\usepackage{orcidlink}
\title[semipositone problems over Heisenberg groups]{ Positive solutions to semipositone problems on Heisenberg group}
 \author[V. Y. Naik and R. Kumar]{Vikram Yallapa Naik$^1$ \orcidlink{0009-0004-6612-2577}  Rohit Kumar$^2$ \orcidlink{0009-0001-6494-6407}} 
 \address{$^1$Department of Mathematics, Birla Institute of Technology and Science, Pilani,\\
 Pilani Campus, Vidya Vihar, Pilani, Rajasthan 333031, India}
 \address{$^2$School of Mathematics and Statistics, Northeast Petroleum University,\\
 	Daqing, 163318, P.R. China.}
 \email{p20210042@pilani.bits-pilani.ac.in, vikramnaik154@gmail.com}  \email{rohit1.iitj@gmail.com}
 \thanks{$^2$Corresponding author}

\begin{document}

\begin{abstract}
In this article, we study the existence of positive weak solutions to a class of semipositone problems on the Heisenberg group $\mathbb{H}^N$. More precisely, we study the problem
\begin{equation*}
          -\Delta_{\mathbb{H}}u=  g(\xi)f_a(u) \text{ in } \mathbb{H}^N \tag{$\mathbb{P}_a$},
\end{equation*} 
where $a>0$ is a real parameter, $g$ is a positive function, and $f_a:\mathbb{R}\to\mathbb{R}$ is a continuous and admits subcritical nonlinearity that takes negative values on part of its domain. The sign-changing nature of $f_a$ causes the strong maximum principle to fail, introducing substantial challenges in establishing positivity. Our first main result concerns the existence of weak solutions to \eqref{Main problem}, which we obtain via the mountain pass theorem. We then derive qualitative properties of these solutions and use them to prove that $u_a \geq 0$ for sufficiently small values of $a$. Finally, by applying the Riesz potential formula, we conclude the strict positivity of solutions. This work extends the results of Alves et. al. (Proc. Roy. Soc. Edinburgh Sect. A, 150(5):2349–2367, 2020) to the setting of the Heisenberg group.
\end{abstract} 
\maketitle
\vspace{-0.5cm}
\section{Introduction}
In this article, we study a class of semipositone problem on the Heisenberg group $\mathbb{H}^N$. The term \textit{semipositone} refers to the behaviour of the source term near zero, which can be strictly negative in certain regions of the domain. To be more precise, we are interested in the following problem: 
\begin{equation} \label{Main problem}
          -\Delta_{\H}u=  g(\xi)f_a(u) \text{ in } \H^N \tag{$\mathbb{P}_a$},
\end{equation} 
where $a>0$ is a real parameter and $f_a: \R \rightarrow \R$ is a continuous function defined as:
\begin{align} \label{1.1}
    f_a(s) = \begin{cases}
        f(s)-a & \text{ if } s \geq 0,\\
        -a(s+1)      & \text{ if }s\in [-1,0],\\
        0      & \text{ if }s \leq -1,
    \end{cases} 
\end{align}
where $f:[0,\infty) \rightarrow [0,\infty)$ is a continuous function such that $f(0)=0$ and $g:\H^N \rightarrow \R$ is a positive function. Further, the functions $g$ and $f$ satisfy the following assumptions:
\begin{enumerate}[label={($\bf g{\arabic*}$)}]
    \setcounter{enumi}{0}
    \item \label{g1} $g \in L^{1}(\H^N) \cap L^{\infty}(\H^N)$. 
   \end{enumerate} 
\begin{enumerate}[label={($\bf f{\arabic*}$)}]
    \setcounter{enumi}{0}
    \item \label{f1}  $\displaystyle \lim\limits_{s \rightarrow 0^+} \frac{f(s)}{s}=0,$ and $\displaystyle \lim\limits_{s \rightarrow \infty} \frac{f(s)}{s^{\gamma-1}} \leq C(f)$ for some $\gamma \in \left(2,Q^*\right)$ and $C(f)>0$. The exponent $ Q^*=\frac{2Q}{Q-2}$ is the Sobolev critical exponent. 
    
    \item \label{f2} (Ambrosetti-Rabinowitz) there exist $\vartheta >2$ and $s_0>0$ such that \[0< \vartheta F(s) \leq sf(s), \quad \forall \, s>s_0, \text{ where } F(s)= \int_{0}^{s} f(\tau) \,\mathrm{d}\tau.\]
   \end{enumerate}
The semipositone problem was first introduced by Castro and Shivaji \cite{Castro1988} in one dimension.
They established the existence of a non-negative solution for the second-order boundary value problem
\begin{align*}
\begin{cases}
    -u''(x) = \lambda f(u(x)), \quad x\in (0,1) \\
    u(0) =u(1)= 0,
\end{cases}
\end{align*}
where $f(0)<0$. Thereafter, Castro and Shivaji\cite{Castro1989} proved the existence of a radial solution to the above problem in the unit ball in dimension $N \geq 2,$ namely,
\begin{align*}
\begin{cases}
    -\Delta u(x) = \lambda f(u(x)), \text{ for } |x| < 1, \quad x \in \mathbb{R}^N, \, \\
u(x) =0 \text{ on }|x|=1.
\end{cases}
\end{align*}
Subsequently, several authors investigated semipositone problems in general bounded domains (cf. \cite{Ambrosetti1994, Castro1995, Caldwell2007, Costa2006}). Using both variational and non-variational techniques, these works established the existence of positive solutions for classical Laplacian problems. Later, the analysis was extended to the classical $p$-Laplacian setting in bounded domains in \cite{Chhetri2015,Castro2016}, who employed the mountain pass theorem and degree theory to obtain positive solutions. In a more general framework, Alves et al. \cite{Alves2019} studied semipositone problems involving nonlinearities with Orlicz growth. Concerning semipositone problems driven by fractional Laplacian operators in bounded domains, the first result was obtained by Dhanya and Tiwari \cite{Dhanya2021}, who treated nonlinearities exhibiting both concave and convex behavior. More recently, Lopera et al. \cite{Lopera2023} addressed fractional $p$-Laplacian semipositone problems with subcritical nonlinearities satisfying Ambrosetti–Rabinowitz type conditions.

To the best of our knowledge, the first study of semipositone problems with classical Laplace operator on the whole space $\mathbb{R}^N$ was carried out by Alves et al. \cite{Alves2020}, where the authors established the existence of a solution via the mountain pass theorem and obtained positivity through the Brézis–Kato regularity method. This line of research was extended to the biharmonic operator in \cite{Biswas2023}. The author in \cite{Biswas2024} investigated the existence of positive solutions for the fractional Laplacian in $\mathbb{R}^N$, and the corresponding fractional $p$-Laplacian case was subsequently addressed in \cite{biswassemipositone}. More recently, Bisci et al. \cite{molica2024mountain} studied semipositone problems involving the subelliptic Grushin operator, again employing the mountain pass theorem to obtain weak solutions.

This article extends the results of \cite{Alves2020} to semipositone problems in the Heisenberg group. The Heisenberg group is a Lie group which has sub-Riemannian geometry. Since the Heisenberg group is non-commutative, it is evident that there are some difficulties in handling problems in such a framework. For futher information, we refer the recent articles \cite{Bai2024,Liang2024,Sun2022} and references therein. More details on the Heisenberg group $\mathbb{H}^N$ is given in section \ref{functional framework}. The existence of a positive solution is not straightforward, since the source term (i.e., $gf_a$ in our case) can be negative in some parts of the domain and we cannot apply the strong maximum principle to prove that the solution is positive. The existence of a positive solution to \eqref{Main problem} depends on the range of the parameter $a>0$. In this direction, we first establish the following theorem, which guarantees the existence of mountain-pass solutions.
\begin{theorem} \label{Existence and uniform boundedness}
Let $f,g$ satisfy \rm{\ref{f1}}-\rm{\ref{f2}} and \rm{\ref{g1}}. Then there exists a weak solution $u_a \in \Dh$ to \eqref{Main problem} for all $a \in (0,a_1)$ and $a_1 >0$ is given in Lemma \ref{MP1}. Moreover, there exists a constant $C>0$ such that $\|u_a\| \leq C$, for all $a \in (0,a_1)$.
\end{theorem}
The details about the solution space $\Dh$ and its associated norm $\|\cdot\|$ are given in Section \ref{functional framework}. In the next result, we prove the regularity of solutions using the Moser iteration technique. In particular, we prove the following theorem.
\begin{theorem}\label{Regularity-of-solutions}
    Let $f$ and $g$ satisfy {\rm\ref{f1}-\ref{f2}} and {\rm\ref{g1}}. There exists  $a_2 \in (0, a_1)$ such that $u_a \in L^\infty(\mathbb{H}^N) \cap C(\mathbb{H}^N)$ for all \( a \in (0, a_2) \).
\end{theorem}
Once we have the above regularity, we proceed to establish the positivity of the solutions and for that purpose we also need the following assumptions:
\begin{enumerate}[label={($\bf f{3}$)}]
    \setcounter{enumi}{0}
    \item \label{f3}  $f$ is locally Lipschitz. 
   \end{enumerate}
\begin{enumerate}[label={($\bf g{2}$)}]
    \setcounter{enumi}{0}
    \item \label{g2} $|\xi|_\mathbb{H}^{Q-2}\int\limits_{\mathbb{H}^N} \frac{g(\eta)}{|\eta^{-1} \circ \xi|_{\mathbb{H}}^{Q-2}} \deta<\infty,$ where $\xi \in \mathbb{H}^N \setminus \{0\}$. 
   \end{enumerate} 
The distance function $|\cdot|_{\mathbb{H}}$ is defined in Section \ref{functional framework}.
\begin{theorem} \label{thm 1.1}
    Let $f$ and $g$ satisfy {\rm\ref{f1}-\ref{f2}, \ref{g1}} and {\rm\ref{g2}}. Then the following results hold.
    \begin{itemize}
        \item[\rm(a)] Then exists $a_3 \in (0,a_2)$ such that $u_a \geq 0$ for all $a \in (0,a_3)$.

        \item[\rm(b)] In addition, if \ref{f3} is satisfied, then there exists $a_4 \in (0,a_3)$ such that $u_a > 0$ for all $a \in (0,a_4)$. 
    \end{itemize}
\end{theorem}

The organization of this article is as follows: In Section \ref{functional framework}, we establish the functional framework and some technical results required for our problem. In Section \ref{existence and regularity}, we prove the existence of the mountain pass solutions depending on the parameter $a$ and subsequently prove that these solutions are uniformly bounded in $\Dh$ with respect to the parameter $a$. In the next step, we prove some qualitative properties of solutions. The final section deals with showing that the mountain pass solutions are positive if $a$ is sufficiently small.

\section{Functional Framework and some technical results} \label{functional framework}
In this section, we establish the functional framework for our problem \eqref{Main problem} and prove several technical lemmas necessary to establish further results. We begin with some preliminaries on the Heisenberg groups. The Heisenberg group $\mathbb{H}^N=\mathbb{R}^{2N+1}$ is the homogeneous stratified Lie group with the group operation $\circ$ defined as
\begin{align*}
\xi\circ\xi'=(x+x',y+y',t+t'+2(x'y-y'x))    
\end{align*}
for every $\xi=(x,y,t),\xi '=(x',y',t')\in \mathbb{H}^N,$ together with the natural group of dilations defined as
\begin{align*}
    \delta_s(\xi)=(sx,sy,s^2t),
\end{align*}
for every $s>0,$ $ x, y \in \mathbb{R}^N$ and $t \in \mathbb{R}$. The homogeneous dimension on $\H^N$ is given by $Q=2N+2$ and the homogeneous norm on $\mathbb{H}^N$ is given by 
\begin{align*}
    |\xi|=|\xi|_{\mathbb{H}}=\left[(x^2+y^2)^2+t^2\right]^{\frac{1}{4}},  \  \mbox{for every}  \   \xi\in \mathbb{H}^N.
\end{align*}
The Lie algebra on $\mathbb{H}^N$ is generated by the vector fields,
\begin{align}\label{vector-fields}
  T=\frac{\partial}{\partial t},  \
X_{j}=\frac{\partial}{\partial
x_{j}}+2y_{j}\frac{\partial}{\partial t},  \
Y_{j}=\frac{\partial}{\partial y_{j}}-2x_{j}\frac{\partial}{\partial
t}, \quad j=1,\ldots,N.  
\end{align}
The horizontal gradient is given by
\begin{align}\label{Horizontal-Grad}
    \nabla_\mathbb{H} u(\xi)=(X_{1}u(\xi),X_{2}u(\xi),\cdots,X_{N}u(\xi),Y_{1}u(\xi),Y_{2}u(\xi),\cdots,Y_{N}u(\xi)).
\end{align}
The sub-elliptic Laplacian $\Delta_\mathbb{H}$ on $\mathbb{H}^N$ is defined as
\begin{align}\label{Sub-elliptic-Laplacian}
    \Delta_{\mathbb{H}}u(\xi)=\sum\limits_{j=1}^{N}X_{j}^{2}u(\xi)+Y_{j}^{2}u(\xi).
\end{align}
The Haar measure on $\H^N$ is the $\R^{2N+1}$ Lebesgue measure. For every $s > 0$, we have  
\[
|\delta_{s}(\Omega)| = s^{Q} |\Omega|, 
\quad \d (\delta_{s}(\xi)) = s^{Q}\, \dxi,
\quad \text{and} \quad 
|B_r(\xi)| = \alpha_{Q} r^{Q},
\]
where $\alpha_{Q} = |B_1(0)|, \ \Omega\subset \mathbb{H^N}.$ Here, $B_r(\xi)$ denotes the ball in 
$\mathbb{H}^N$ centered at $\xi$ with radius $r$. Polar integration formula on \( \mathbb{H}^N \) (see \cite[Proposition 1.15]{folland1982hardy}) is given by 
$$
\int\limits_{\mathbb{H}^N} f(\xi) \dxi = \int\limits_{0}^\infty \int\limits_{\Sigma} f(\delta_r(\xi)) r^{Q-1} \d \sigma\dr,
$$
where $f$ is an integrable function, $\Sigma= \{\xi \in \mathbb{H}^N : \quad \left|\xi \right|_{\H} =1\}$ and $\d \sigma$ is the surface measure.
For more details, we refer to the monograph \cite{Bonfiglioli2007}.

The natural space to look for the solutions to \eqref{Main problem} is the Beppo Levi space $\Dh$, which is the completion of $C_c^\infty(\mathbb{H}^N)$ with respect to the norm 
\begin{align}\label{norm}
   \norm{u}:=\left( \int_{\H^N}\abs{\nabla_\H u}^2 \,\dxi \right)^{\frac{1}{2}}.
\end{align}
The space $\Dh$ is a Hilbert space and it can also be characterized as follows
\begin{align}\label{solution-space}
    \Dh = \left\{ u \in L^{Q^*}(\H^N) : \int_{\H^N}\abs{\nabla_\H u}^2 \,\dxi< \infty  \right\}.
\end{align}
A function $u \in \Dh$ is said to be a weak solution to \eqref{Main problem} if the following Euler-Lagrange equation
\begin{align}\label{weak formulation}
    \int_{\H^N} \nabla_\H u \cdot \nabla_\H \varphi\, \dxi = \int_{\H^N} g(\xi)f_a(u) \varphi(\xi)\,\dxi,
\end{align}
holds for every $\varphi \in \Dh$. The energy functional associated with the weak formulation \eqref{weak formulation} is given by
\begin{align}\label{energy functional}
    J_a(u) = \frac{1}{2} \int_{\H^N}\abs{\nabla_{\H}u}^2\,\dxi - \int_{\H^N} g(\xi)F_a(u)\,\dxi.
\end{align}
If we define $K_a: \Dh \rightarrow \R$ as $ K_a(u) := \int_{\H^N} g(\xi)F_a(u(\xi))\,\dxi$. Then, $K_a$ is continuously Fr\'{e}chet differentiable on $\Dh$, and consequently $J_a$ becomes continuously Fr\'{e}chet differentiable on $\Dh$. The Fr\'{e}chet derivatives of $J_a$ is given by
 \begin{equation}\label{2.1}
     \begin{split}
   \langle J_a'(u), v\rangle &= \int_{\H^{N}} \nabla_\H u \cdot \nabla_\H v\, \dxi -\langle K_a'(u), v \rangle,\\
    &= \int_{\H^{N}} \nabla_\H u \cdot \nabla_\H v\, \dxi -\int_{\H^N} g(\xi)f_a(u) v\,\dxi,\; \forall \, v \in \Dh. 
     \end{split}
 \end{equation}
Thus, the critical points of $J_a$ are weak solutions to \eqref{Main problem} and our goal is to show the existence of positive critical points. 

From the hypotheses \ref{f1} and \ref{f2}, we derive the following estimates on $f_a$ and its associated function $F_a$ defined as $F_a(u):=\int_{0}^{u}f_a(\tau)\,\mathrm{d} \tau$:
\begin{align}\label{2.2}
     |f_a(s)| \leq \epsilon |s| + C(f,\epsilon) |s|^{\gamma-1} +a \text{ and } |F_a(s)| \leq \epsilon |s|^{2} + C(f,\epsilon) |s|^{\gamma}+a|s| \text{ for } s \in \R. 
\end{align} 
Moreover, for some $\tilde{a} \in (0,a)$, we can write
\begin{align}\label{2.3}
    |f_a(s)| \leq C(1+ |s|^{\gamma-1}) \text{ and } |F_a(s)| \leq C(|s|+|s|^\gamma) \text{ for } s \in \R, \text{ where } C=C(f,\epsilon,\Tilde{a}). 
\end{align}
Using \ref{f2}, we can find some $M_1,M_2>0$ such that the following estimate of $F$ holds:
\begin{align} \label{2.4}
    F(s) \geq M_1 s^\vartheta - M_2, \; \forall \, s \geq 0.  
\end{align}
From the Ambrosetti-Rabinowitz type condition \ref{f2} on $f$, we can also establish the following Ambrosetti-Rabinowitz condition on $f_a$ as follows:
\begin{align} \label{2.5}
    \vartheta F_a(s) \leq s f_a(s)+M_3,\;  \forall \, s \in \R \text{ and } a \in (0,\tilde{a}),
\end{align} 
where $M_3$ is independent of $a$ and $s$.

The weighted Lebesgue space $L^q(\H^N,g)$ is defined as
\begin{align*}
    L^q(\H^N,g) = \left\{ u \text{ is Lebesgue measurable} : \int_{\H^N}g(\xi)|u(\xi)|^q\,\dxi<\infty \right\}.
\end{align*} 
Recall that the embedding $\Dh \hookrightarrow L^{Q^*}(\H^N)$ is continuous and $\Dh \hookrightarrow L_{\text{loc}}^{q}(\H^N)$ is embedded continuously and compactly for every $q \in [1,Q^{*})$ (see \cite[Theorem 2.1]{Sun2022}). In the next result, we prove that $\Dh$ is compactly embedded in a weighted Lebesgue space.
\begin{lemma}\label{compactness weight g} 
    Let $g \in L^{r}(\H^N)$, where $r=\frac{Q^*}{Q^*-q}, q\in [1,Q^*)$. Then $\Dh$ is compactly embedded into $L^{q}(\H^N, g)$.
\end{lemma}
\begin{proof}
Let $\{u_n\}$ be any bounded sequence in $\Dh$. By reflexivity, $u_n \rightharpoonup u$ in $\Dh$ up to a subsequence for some $u \in \Dh$. By density, for every $\epsilon>0$ there exists a $g_\epsilon \in C_c(\H^N)$ (i.e., the space of continuous functions with compact support) such that
\begin{align*}
    \|g-g_\epsilon\|_{L^r}<\frac{\epsilon}{2A}, \text{ where } A:= \sup\limits_{n \in \N}\|u_n-u\, \|_{L^{Q^{*}}}^q.
\end{align*}
Assume that $B$ is the upper bound of $g_\epsilon$ on the compact set $\text{supp}(g_\epsilon)$. Then, using H\"{o}lder inequality, we get
\begin{align}\label{2.6}
     \int_{\H^N} g|u_n-u|^q\,\dxi 
    & \leq \|g-g_\epsilon\|_{L^r} \|u_n-u\|_{L^{Q^*}}^q + \int_{\text{supp}(g_\epsilon)}|g_\epsilon||u_n-u|^q\,\dxi \notag\\
    & \leq \frac{\epsilon}{2} + B\int_{\text{supp}(g_\epsilon)} |u_n-u|^q\,\dxi.
\end{align}
The second integral tends to zero as $n \rightarrow \infty$ using the compact embedding $\Dh \hookrightarrow L^{q}_{\text{loc}}(\H^N)$. Taking the limit as $n \rightarrow \infty$ into \eqref{2.6} and since $\epsilon >0$ is arbitrary, we conclude that the proof.
\end{proof}
\begin{remark}
    Note that if $g \in L^1(\H^N) \cap L^\infty(\H^N)$, then $g \in L^r(\H^N)$ for every $r\geq 1$. Hence, if $g$ satisfies {\rm\ref{g1}}, the embedding $\Dh \hookrightarrow L^{q}(\H^N,g), q\in [1,Q^*),$ is compact using Lemma \ref{compactness weight g}. 
\end{remark} 
Further, we prove that the functional $J_a$ admits a mountain pass geometry.
\begin{lemma}\label{MP1}
Let $f$ satisfies {\rm\ref{f1}}, {\rm \ref{f2}} and $g$ satisfies {\rm \ref{g1}}. Then
\begin{enumerate}[label=\rm(\alph*)]
       \item there exist $\beta, \delta, a_1>0$ such that $J_a(u) \geq \delta$ on $\norm{u}=\beta$ and $a \in (0,a_1)$.
       \item there exists $v \in \Dh$ such that $\|v\|>\beta$ and $J_a(v)<0$. 
   \end{enumerate} 
\end{lemma}
\begin{proof}
  (a) The estimate on $F_a$ given by \eqref{2.2} yields
\begin{align*}
    \int_{\H^N} g(\xi)F_a(u)\,\dxi \leq \int_{\H^N} g(\xi)\left(\epsilon |u|^2+ C(f,\epsilon) |u|^\gamma+ a|u| \right)\, \dxi \leq \epsilon C_1 \norm{u}^2 + C(f,\epsilon) C_2 \norm{u}^\gamma + a C_3 \|u\|,
\end{align*}  
where $C_1, C_2$ and $C_3$ are the embedding constants for $\Dh \hookrightarrow L^q( \H^N,g)$ (see Lemma \ref{compactness weight g}). For $\norm{u}=\beta$, we have
\begin{align}\label{2.12}
    J_a(u) \geq \frac{1}{2} \beta^2 - \epsilon C_1 \beta^2 - C(f,\epsilon) C_2 \beta^\gamma - a C_3 \beta = \beta^2 \left(\frac{1}{2} - \epsilon C_1 - C(f,\epsilon) C_2 \beta^{\gamma-2} \right)-a C_3 \beta.
\end{align}
We choose $\epsilon < (2C_1)^{-1}$. Now we can write $J_a(u) \geq A_1(\beta) - aC_3 \beta$, where $A_1(\beta) = C \beta^2 (1-\widetilde{C}\beta^{\gamma -2})$ with $C,\widetilde{C}$ independent of $a$. If we assume $\beta_1$ to be the non-trivial zero of $A_1$, then for $\beta < \beta_1$, we set $a_1 \in (0, \frac{A_1(\beta)}{C_3 \beta})$ and $\delta = A_1(\beta)-a_1 C_3 \beta$. Using \eqref{2.12}, we conclude that $J_a(u) \geq \delta$ for every $a \in (0,a_1)$. 

\noi(b) Assume that $\varphi \in C_c^\infty(\H^N) \setminus \{0\} \text{ with } \varphi \geq 0$ and $\|\varphi\|=1$. Taking $s \geq 0$, we write
    \[ J_a(s \varphi) = \frac{s^2}{2} \int_{\H^{N}}\,\abs{\nabla_\H \varphi}^2 \dxi
    - \int_{\H^N}g(\xi) \left(F(s \varphi) -as\varphi  \right) \,\dxi.\]
We use the Ambrosetti-Rabinowitz condition \eqref{2.4} of $F$ to establish the following estimates:
\begin{align}\label{2.13}
    J_a(s \varphi) &\leq \frac{s^2}{2} \|\varphi\|^2
    -M_1s^\vartheta \int_{\H^N}g(\xi)(\varphi(\xi))^\vartheta\,\dxi+ M_2 \int_{\H^N}g(\xi)\,\dxi+ as \int_{\H^N}g(\xi) \varphi(\xi)\,\dxi \notag\\
    &\leq  \frac{s^2}{2} -M_1 s^\vartheta \int_{\H^N}g(\xi)(\varphi(\xi))^\vartheta\,\dxi+ M_2 \|g\|_{L^1}  + as \norm{\varphi}_{L^{\infty}} \norm{g}_{L^1}.
\end{align}
Since $\vartheta >2$, we observe that $J_a(s \varphi) \rightarrow - \infty$ as $s \rightarrow +\infty$. Therefore, we can find a number $s_1$ such that $s_1 > \beta$ and $J_a(s\varphi)<0$ for $s>s_1$. Choosing $v= s\varphi$ for $s>s_1$ gives the required function.
\end{proof}

\begin{definition}
Let $X$ be a Banach space and $J:X \rightarrow \R$ be a continuously Fr\'{e}chet differentiable functional. Then, we call a sequence $\{u_n\}$ to be Palais-Smale for $J$ if $\{J(u_n)\}$ is bounded in $\R$ and $J'(u_n) \rightarrow 0$ in $X^*$ (dual space of $X$). Further, $J$ satisfies the Palais-Smale condition (in short (PS) condition) if every Palais-Smale sequence for $J$ has a convergent subsequence.
\end{definition} 

\begin{proposition} \label{Palais smale J_a}
Let $f$ and $g$ satisfy \rm{\ref{f1}}, \rm{\ref{f2}} and $g$. Then $J_a$ admits the Palais-Smale condition for every $a \geq 0$.
\end{proposition}
\begin{proof}
Let $\{u_n\}$ be a Palais Smale sequence for $J_a$. In order to prove $\{u_n\}$ has a strongly convergent subsequence in $\Dh$, first we claim that $\{u_n\}$ is bounded in $\Dh$. Notice that $|J_a(u_n)|\leq M$ for some $M>0$ and subsequently we have
\begin{equation}\label{2.14}
    \frac{1}{2}\norm{u_n}^2 - \int_{\H^N} g(\xi)F_a(u_n)\,\dxi\leq M, \; \text{ for all } \, n\in \N.
\end{equation}
Next, using $J_a'(u_n) \rightarrow 0$ in $(\Dh)^*$, we can find $n_1 \in \N$ such that $\abs{\langle J_a'(u_n), u_n \rangle} \leq \norm{u_n}$, for all $n \geq n_1$. Hence, we get
\begin{equation}\label{2.15}
    -\norm{u_n} -\norm{u_n}^2 \leq -\int_{\H^N} g(\xi)f_a(u_n)u_n\,\dxi, \; \text{ for all } \, n \geq n_1.
\end{equation}
Combining \eqref{2.5} and \eqref{2.14} gives
\begin{equation}\label{2.16}
    \frac{1}{2}\norm{u_n}^2 - \frac{1}{\vartheta}\int_{\H^N} g(\xi)f_a(u_n)u_n\,\dxi- \frac{1}{\vartheta} M_3\norm{g}_{L^1}\leq M,
\end{equation}
and further by \eqref{2.15} and \eqref{2.16},
\begin{align*}
    \bigg(\frac{1}{2}-\frac{1}{\vartheta} \bigg)\norm{u_n}^2 -\frac{1}{\vartheta}\norm{u_n} \leq M + \frac{1}{\vartheta}M_3 \norm{g}_{L^1}, \; \text{ for all } \, n \geq n_1.
\end{align*}
From the above inequality, we deduce that $\{u_n\}$ is bounded in $\Dh$. Hence, there exists some $u \in \Dh$ such that up to a subsequence $u_n \rightharpoonup u$ in $\Dh$. We define a functional $J: \Dh \rightarrow \R$ as $J(u) = \frac{1}{2}\|u\|^2$. Notice that $J$ is continuously differentiable on $\Dh$. Moreover, we can write 
\begin{align*}
    \langle J'(u_n), u_n -u \rangle = A_n + B_n.
\end{align*}
where $A_n = \langle J_a'(u_n),u_n-u \rangle$ and $B_n= \langle  K_a'(u_n), u_n-u \rangle$. Next we claim that $ \langle J'(u_n), u_n -u \rangle \rightarrow 0$ as $n \rightarrow \infty$. It is enough to show that $A_n, B_n \rightarrow 0$ as $n \rightarrow \infty$. Since $J_a'(u_n) \rightarrow 0 \text{ in } (\Dh)^*$ and $\{u_n\}$ is bounded in $\Dh$, we have $ \abs{A_n}  \leq \|J_a'(u_n)\|_{*} \|u_n-u\| \rightarrow 0$ as $n \rightarrow \infty$. Consequently, $A_n \rightarrow 0$ as $n \rightarrow \infty$. Next, we show $B_n \rightarrow 0$ as $n \rightarrow \infty$. Using \eqref{2.3}, we have
 \begin{align} \label{2.17}
     |B_n| & \leq \int_{\H^N} g(\xi)|f_a(u_n)||u_n-u|\, \dxi  \leq C(f,a) \int_{\H^N} g(\xi) \left( 1 + |u_n|^{\gamma-1} \right) |u_n-u|\, \dxi.
 \end{align} 
 The sequence $\{|u_n|^{\gamma -1}\}$ is bounded in $L^{\frac{Q^*}{\gamma-1}}(\H^N)$. Further, 
 \begin{align*}
     \gamma < Q^* \Longleftrightarrow \frac{Q^*}{Q^*-(\gamma-1)} < Q^*.
 \end{align*}
 From Lemma \ref{compactness weight g} and {\rm \ref{g1}}, we get $u_n \ra u$ in $L^{\frac{Q^*}{Q^*-(\gamma-1)}}(\H^N,g)$. Further, the H\"{o}lder's inequality with conjugate pair $(\frac{Q^*}{\gamma-1}, \frac{Q^*}{Q^*-(\gamma-1)})$ gives
 \begin{align}
     \int_{\H^N} g(\xi)|u_n-u||u_n|^{\gamma-1}\,\dxi\leq \norm{g}_{L^\infty}^{\frac{\gamma-1}{Q^*}} \norm{u_n-u}_{L^{\frac{Q^*}{Q^*-(\gamma-1)}}(\H^N,g)} \||u_n|^{\gamma-1}\|_{L^\frac{Q^*}{\gamma-1}} \rightarrow 0 \text{ as } n \rightarrow \infty.
 \end{align} 
Moreover, Lemma \ref{compactness weight g} infers that $\int_{\H^N} g(\xi)|u_n-u|\,\dxi\rightarrow 0$ as $n \ra \infty$. Consequently, we get
 \begin{align}\label{2.19}
     \int_{\H^N} g(\xi)\left( 1 + |u_n|^{\gamma-1}   \right)|u_n-u| \,\dxi\rightarrow 0 \text{ as } n \rightarrow \infty.
 \end{align}
From \eqref{2.17} and \eqref{2.19}, we deduce that $B_n \rightarrow 0$ as $n \rightarrow \infty$. Since $J \in \C^1(\Dh, \R)$ and $u_n \rightharpoonup u$ in $\Dh$, we also get $\langle J'(u), u_n -u \rangle \rightarrow 0$ as $n \rightarrow \infty$. Using H\"{o}lder inequality, we estimate
\begin{align}\label{2.20}
     \langle J'(u_n), u_n -u \rangle - \langle J'(u), u_n -u \rangle &= \int_{\H^N} |\nabla_\H u_n|^2 \, \dxi + \int_{\H^N} |\nabla_\H u|^2 \, \dxi - 2\int_{\H^N} \nabla_\H u \cdot \nabla_\H u_n\, \dxi  \notag\\
     & \geq \norm{u_n}^2 + \norm{u}^2 - 2\norm{u_n}\norm{u} \notag\\
     & \geq 0.
 \end{align}
Taking limit as $n \rightarrow \infty$ in \eqref{2.20}, we get
 $\norm{u_n} \rightarrow \norm{u}$. Hence, we conclude
 $u_n \rightarrow u$ in $\Dh$. 
\end{proof}
\section{Existence and regularity of solutions} \label{existence and regularity}
\begin{proof}[Proof of Theorem \ref{Existence and uniform boundedness}]
We choose $\delta,v$ given in Lemma \ref{MP1}. Notice that $J_a$ satisfies all the hypotheses of the mountain pass theorem \cite[Theorem 2.1]{Ambrosetti1973} for every $a \in (0,a_1)$ due to Lemma \ref{MP1} and Proposition \ref{Palais smale J_a}. Hence, $J_a$ admits a mountain pass critical point $u_a \in \Dh$ such that
   \begin{equation} \label{3.1}
       J_a(u_a) = c_a = \inf\limits_{\gamma \in \Gamma_v}\max\limits_{s \in [0,1]}J_a(\gamma(s)) \geq \delta \text{  and  } J_a'(u_a)=0,
   \end{equation}
where $\Gamma_v := \{ \gamma \in C([0,1], \Dh):\gamma(0)=0  \text{ and } \gamma(1)=v \}$ and $c_a$ is the mountain pass level associated with $J_a$. Now we prove the uniform boundedness of solutions $\{u_a\}$ in $\Dh$. In order to prove the uniform boundedness, first, we need to establish the uniform boundedness of $\{J_a(u_a): a \in (0,a_1)\}$. We define a path $\Tilde{\gamma}:[0,1] \rightarrow \Dh$ by $\Tilde{\gamma}(\sigma)=\sigma v$, where $v=s\varphi$ for some $s>s_1$ (for $s_1$ as in Lemma \ref{MP1}-(b)), $\varphi \in \C_c^{\infty}(\H^N)\setminus\{0\},~ \varphi \geq 0$ and $\norm{\varphi}=1$. Observe that $\Tilde{\gamma} \in \Gamma_v$ since $\Tilde{\gamma}(0)=0$ and $\Tilde{\gamma}(1)=v$. We obtain from \eqref{3.1} that
\begin{equation} \label{3.2}
       J_a(u_a) = c_a = \inf\limits_{\gamma \in \Gamma_v}\max\limits_{\sigma \in [0,1]}J_a(\gamma(\sigma)) \leq \max\limits_{\sigma \in [0,1]}J_a(\Tilde{\gamma}(\sigma)) = \max\limits_{\sigma \in [0,1]}J_a(\sigma s\varphi).
   \end{equation}
Following the calculations as in \eqref{2.13}, we deduce that
 \begin{align}\label{3.3}
    J_a(\sigma s \varphi) &\leq \frac{{\sigma}^2s^2}{2} \norm{\varphi}^2
    -M_1 {\sigma}^\vartheta s^\vartheta \int_{\H^N}g(\xi)(\varphi(x))^\vartheta\,\dxi + M_2 \int_{\H^N}g(\xi)\,\dxi+ a {\sigma} s \int_{\H^N}g(\xi) \varphi(x)\,\dxi  \notag\\
    &\leq  \frac{s^2}{2} + M_2 \|g\|_{L^1}  + a_1 s C_1 \norm{\varphi}.
\end{align}
From \eqref{3.2} and \eqref{3.3}, there exists $C= C(N,s,M_2,g,a_1)$ satisfying
 \begin{equation}\label{3.4}
     J_a(u_a) \leq C, \text{ for all } a \in (0,a_1).
 \end{equation}
From \eqref{3.1}, we write
 \begin{equation}\label{3.5}
     \norm{u_a}^2 - \int_{\H^N}g(\xi)f_a(u_a)u_a\,\dxi = 0.
 \end{equation}
 Also, by \eqref{3.4} we have
 \begin{equation}\label{3.6}
     \frac{1}{2}\norm{u_a}^2 - \int_{\H^N}g(\xi)F_a(u_a)\,\dxi \leq C.
 \end{equation}
From \eqref{3.5} and \eqref{3.6}, we deduce the following
\begin{align*}
    \left(\frac{1}{2}-\frac{1}{\vartheta}\right)\norm{u_a}^2 + \int_{\H^N}g(\xi) \left(\frac{1}{\vartheta}f_a(u_a)u_a- F_a(u_a) \right)\,\dxi \leq C.
\end{align*}
We combine the above estimate with \eqref{2.5} to get
\begin{align*}
    \left(\frac{1}{2}-\frac{1}{\vartheta}\right)\norm{u_a}^2 -\frac{1}{\vartheta}M_3\|g\|_{L^1} \leq C.
\end{align*}
Hence, we can conclude that $\norm{u_a} \leq C$ for every $a \in (0,a_1)$, where $C=C(N,s,g,M_2,M_3,a_1)$.
\end{proof} 

\begin{proof}[Proof of Theorem \ref{Regularity-of-solutions}]
We only need to show that for any sequence $a_j \rightarrow 0$, the sequence of solutions $\{u_j=u_{a_j}\}$ possesses a subsequence that is bounded in $L^\infty(\H^N)$. Since $\{u_j\}$ is bounded in $\Dh$ (See Theorem \ref{Existence and uniform boundedness}), there exists some $u \in \Dh$ such that
\begin{align*}
    u_j \rightharpoonup u \quad &\text{ in } \Dh,\\
    u_j \rightarrow u \quad &\text{ a.e in } \H^N.
\end{align*}
Following similar arguments as in Proposition \ref{Palais smale J_a}, we deduce that $u_j \rightarrow u$ in $\Dh$ and consequently $u_j \rightarrow u$ in $L^{Q^*}(\H^N)$ by the Sobolev embedding $\Dh \hookrightarrow L^{Q^*}(\H^N)$. Moreover, there is some $h \in L^{Q^*}(\H^N)$ such that up to a subsequence we have
\begin{align*}
    |u_j(\xi)| \leq h(\xi) \text{ a.e. in } \H^N \text{ and for all } j \in \mathbb{N}. 
\end{align*}
We define $V_j(\xi) = g(\xi) \tfrac{\left( 1+ |u_j|^{\frac{Q+2}{Q-2}}\right)}{1+|u_j|}$. Then we have
\begin{align*}
    |V_j(\xi)| \leq g(\xi)\left(1+ |u_j(\xi)|^{\frac{4}{Q-2}} \right) \leq g(\xi)\left(1+ |h(\xi)|^{\frac{4}{Q-2}} \right).
\end{align*}
Since $h \in L^{Q^*}(\H^N)$, we obtain $V_j \in L^{\frac{Q}{2}}(\H^N)$. Moreover, $V_j \in L^1(\H^N)$. Therefore, the interpolation inequality yields
\begin{align*}
    V_j \in L^r(\H^N) \text{ for every } r \in \left[1, \frac{Q}{2} \right].
\end{align*}
By the dominated convergence theorem, we get
$V_j \rightarrow V$ in $L^r(\H^N)$ for $r \in \left[1, \frac{Q}{2} \right]$. Notice that
\begin{align*}
      |g(\xi)f_{a_j}(u_j)| \leq |g(\xi)|\left(\varepsilon |u_j|  + C_\varepsilon |u_j|^{\frac{Q+2}{Q-2}} + a_j\right).
\end{align*}
Since $V_j \rightarrow V$ a.e. in $\H^N$, there exists a $j_0$ large enough such that for $j \geq j_0,$ we have
\begin{align*}
     |g(\xi)f_{a_j}(u_j)| &\leq C|g(\xi)| \left( 1 + |u_j|^{\frac{Q+2}{Q-2}} \right) \leq C V(\xi) \left( 1 + |u_j| \right).
\end{align*}
We claim that  for any $q \in [Q^*, \infty),$ there exists a constant $K_q$ such that $\|u_j\|_{L^q} \leq K_q$ for all $j \geq j_0$. First we prove that if $u_j \in L^q(\mathbb{H}^N)$ for some $q \geq Q^*,$ then $u_j \in L^{\frac{Q^*}{2}q}(\H^N)$. So, we assume that $u_j \in L^q(\mathbb{H}^N)$. Notice that we can write $u_j=u_j^+ -u_j^-$, where $u_j^+:= \max\{u_j,0\}$ and $u_j^-:= \max\{-u_j,0\}$. For $L \geq 0$, we define $G(u_j^+) := G_L(u_j^+)=u_j^+ \min\{(u_j^+)^{\frac{q}{2}-1}, L\}$, $F(u_j^+):=F_L(u_j^+)=\int\limits_0^{u_j^+} |G'(s)|^2\,\mathrm{d} s$. Then by definition \eqref{vector-fields}, we have
\begin{align*}
    X_iF(u_j^+)&=|G'(u_j^+)|^2X_iu_j^+, \\
    Y_iF(u_j^+)&=|G'(u_j^+)|^2Y_iu_j^+,
\end{align*}
where $X_i$ and $Y_i$ are vector fields as given in \eqref{vector-fields}. Using the definition \eqref{Horizontal-Grad}, we get $\nabla_\mathbb{H} F(u_j^+) = |G'(u_j^+)|^2\nabla_\mathbb{H} u_j^+$, and
\begin{align}\label{rev-e1}
    \nabla_\mathbb{H} F(u_j^+) \cdot \nabla_\mathbb{H}u_j^+ &= |G'(u_j^+)|^2 |\nabla_\mathbb{H} u_j^+|^2 \notag\\
    &= |G'(u_j^+) \nabla_\mathbb{H} u_j^+|^2 \notag\\
    &= |\nabla_\mathbb{H}G(u_j^+)|^2.
\end{align}
Using \eqref{rev-e1}, we get
\begin{align*}
\int\limits_{\mathbb{H}^N}\nabla_{\mathbb{H}} u_j \cdot \nabla_{\mathbb{H}}F(u_j^+)\, \dxi
&= \int\limits_{\mathbb{H}^N}\nabla_{\mathbb{H}} u_j^+ \cdot \nabla_{\mathbb{H}}F(u_j^+)\, \dxi - \int\limits_{\mathbb{H}^N}\nabla_{\mathbb{H}} u_j^- \cdot \nabla_{\mathbb{H}}F(u_j^+)\, \dxi\\
&= \int\limits_{\mathbb{H}^N} |\nabla_{\mathbb{H}}G(u_j^+)|^2\, \dxi.
\end{align*} 
The second integral on the right hand side in the first equality vanishes since $u_j^+$ and $u_j^-$ can not be simultaneously non-zero.
Moreover, we have the following estimate
\begin{align*}
    u_j^+ F(u_j^+)\leq C_q (G(u_j^+))^2 \leq  C_q  (u_j^+)^q,
 \end{align*}
 where $C_q= {\frac{q^2}{4(q-1)}}$.  
Taking $F(u_j^+)$ as a test function and $j \geq j_0$, we get
\begin{align*}
    \int\limits_{\mathbb{H}^N} \nabla_\mathbb{H} u_j \cdot \nabla_\mathbb{H}F(u_j^+) d\xi  
    &\leq C\int\limits_{\mathbb{H}^N} V( \xi) \left(1 + |u_j|\right) u_j^+ \min \left\{(u_j^+)^{q-2}, L^2 \right\}   \dxi \\
&\leq C\int\limits_{\mathbb{H}^N} V( \xi) \left(1 + 2 (u_j^+)^2 \right) \min \left\{(u_j^+)^{q-2}, L^2 \right\}   \dxi \\
&\leq C\int\limits_{\mathbb{H}^N} V( \xi)   \dxi + 3 \int\limits_{\mathbb{H}^N} V( \xi) (u_j^+)^2 \min \left\{(u_j^+)^{q-2}, L^2 \right\}   \dxi.
\end{align*}
Since $u_j^+ \in L^q(\mathbb{H}^N)$, then for any  $K \geq 1 $
\begin{align*}
\int\limits_{\mathbb{H}^N} |\nabla_\mathbb{H} G(u_j^+)|^2   \dxi &\leq C + 3 \int\limits_{\{ V(\xi) \leq K\}} V(\xi) (u_j^+)^2 \min \left\{ (u_j^+)^{q-2}, L^2 \right\}   \dxi + 3 \int\limits_{\{ V(\xi) \geq K\}} V(\xi) (u_j^+)^2 \min \left\{ (u_j^+)^{q-2}, L^2 \right\}   \dxi\\
 &\leq C + 3K \int\limits_{\{ V(\xi) \leq K\}} (u_j^+)^q  \dxi + 3 \int\limits_{\{ V(\xi) \geq K\}} V(\xi) (u_j^+)^2 \min \left\{ (u_j^+)^{q-2}, L^2 \right\}   \dxi\\
 &\leq C(1 + K) + C \left(\int\limits_{\{ V(\xi) \geq K\}} V(\xi)^{\frac{Q}{2}} \dxi \right)^{\frac{2}{Q}} \left( \int\limits_{\{ V(\xi) \geq K\}} \left((u_j^+) \min \left\{ (u_j^+)^{\frac{q}{2}-1}, L \right\}\right)^{\frac{2Q}{Q-2}}   \dxi\right)^{\frac{Q-2}{Q}}\\
 &\leq C(1 + K) + C \left(\int\limits_{\{ V(\xi) \geq K\}} V(\xi)^{\frac{Q}{2}} \dxi \right)^{\frac{2}{Q}} \left( \int\limits_{\{ V(\xi) \geq K\}} |G(u_j^+)|^{\frac{2Q}{Q-2}}   \dxi\right)^{\frac{Q-2}{Q}}\\
 &\leq  C(1 + K) + C \varepsilon(K) \int\limits_{\mathbb{H}^N}  |\nabla_{\mathbb{H}}G(u_j^+)|^2 \dxi,
\end{align*}
where $\varepsilon(K) := \left(\int\limits_{\{ V(\xi) \geq K\}} V(\xi)^{\frac{Q}{2}} \dxi \right)^{\frac{2}{Q}}.$ We have that $\varepsilon(K) \to 0$ as $K \to \infty.$ Fix $K$ such that $\varepsilon(K) \leq \frac{1}{2C}$ so that 
\begin{align}\label{monotone bound}
    \int\limits_{\mathbb{H}^N} |\nabla_\mathbb{H} G(u_j^+)|^2   \dxi \leq C,
\end{align}
where $C$ depends on $\|u_j^+\|_{L^q}$ which is independent of $L$. Since $G(u_j^+)$ is non-decreasing function with respect to $L$ and \eqref{monotone bound} holds, the monotone convergence theorem yields $G(u_j^+) \to (u_j^+)^\frac{q}{2}$ in $\Dh$ as $L \rightarrow \infty$. Due to the embedding $\Dh \hookrightarrow L^{Q^*}(\H^N)$, we get $u_j^+ \in L^{\frac{Q^*}{2}q}(\H^N)$. Similarly, if we replace $u_j^+$ with $u_j^-$ in the definition of $G(u_j^+)$ and $F(u_j^+)$, we obtain that $u_j^- \in L^{\frac{Q^*}{2}q}(\H^N)$. Hence, we deduce that $u_j \in L^{\frac{Q^*}{2}q}(\H^N)$. Now iterating the same process for different $q$ starting from $Q^*$, we obtain $u_j\in L^q(\mathbb{H}^N)$ for all $q\in [Q^*,\infty)$.

We now claim that there exists $C>0$ such that $\norm{u_j}_{L^\infty}\le C$ for all $j\ge j_0$.
By the previous claim, we have that $u_j\in L^q(\mathbb{H}^N)$ for all $q\in [Q^*,\infty)$.
Consider,
\begin{align*}
    \int\limits_{\mathbb{H}^N}  V(\xi)^q \dxi &\leq C \left(\int\limits_{\mathbb{H}^N} g(\xi)^q \dxi + \int\limits_{\mathbb{H}^N} g(\xi)^q|u_j|^{\frac{Q+2}{Q-2}q} \dxi \right) \leq C,
\end{align*}
where we have used \ref{g1} and H\"{o}lder's inequality.
Hence, $V\in L^q(\mathbb{H}^N), q\in [1,\infty)$. Now, we again argue for $u_j^+$. Let $s_0>Q/2$. We construct a sequence $\|u_j^+\|_{L^{q_k}}$ with $q_k\to\infty$, which are uniformly bounded by $\|u_j^+\|_{L^{q_0}}$, where $q_0 = s_0'Q^*$ and $s_0'$ is such that $\frac{1}{s_0}+ \frac{1}{s_0'} = 1$. As in the previous step, again using the test function $F(u_j^+),$ we have
\begin{align*}
\int\limits_{\mathbb{H}^N} V(\xi)(1+|u_j|)F(u_j^+) d\xi & \textcolor{magenta}{=} \int\limits_{\mathbb{H}^N} V(\xi)|u_j|F(u_j^+) \dxi + \int\limits_{\mathbb{H}^N} V(\xi)F(u_j^+) \dxi \\
&\leq \|V\|_{L^{s_0}} \||u_j|F(u_j^+)\|_{L^{s'_0}} +C_q \int\limits_{\mathbb{H}^N}  V(\xi) (u_j^+)^{q-1} d\xi \\
&\leq C_q\|V\|_{L^{s_0}} \|u_j^+\|_{L^{qs'_0}}^q + C_q\|V\|_{L^{\tilde{s}_0}}\|u_j^+\|_{L^{qs'_0}}^{q-1} \text{ where } \tilde{s}_0=\frac{q s_0}{q s_0 -q + 1} > \frac{Q}{2} \\
&\leq C_q\left(\|V\|_{L^{s_0}} + \|V\|_{L^{\tilde{s}_0}}\|u_j^+\|_{L^{qs'_0}}^{-1}\right)\|u_j^+\|_{L^{qs'_0}}^{q} \\
&\leq \tilde{C}(q,V,u_j)C_q\|V\|_{L^{s_0}}\|u_j^+\|_{L^{qs'_0}}^{q}, 
\end{align*}
which implies
\begin{align*}
\int\limits_{\mathbb{H}^N} |\nabla_\mathbb{H} G(u_j^+)|^2   \dxi \leq \int\limits_{\mathbb{H}^N} V(\xi)(1+|u_j|)F(u_j^+) d\xi \leq \tilde{C}C_q\|V\|_{L^{s_0}}\|u_j^+\|_{L^{qs'_0}}^{q}.
\end{align*}
The Sobolev inequality gives,
\begin{align*}
    \|G(u_j^+)\|^2_{L^{Q^*}} \leq \tilde{C}C_q\|V\|_{L^{s_0}}\|u_j^+\|_{L^{qs'_0}}^{q}.
\end{align*}
Taking $L \to \infty$ and applying the Monotone convergence theorem, we get
\begin{align*}
    \|u_j^+\|^q_\frac{Q^*q}{2} \leq C C_q\|V\|_{L^{s_0}}\|u_j^+\|_{L^{qs'_0}}^{q}.
\end{align*}
Taking $\theta = \frac{Q^*}{2s'_0} > 1$, we get
\begin{align*}
    \|u_j^+\|_{L^{\theta s'_0 q}} \leq (C C_q)^{\frac{1}{q}}\|V\|^{\frac{1}{q}}_{L^{s_0}}\|u_j^+\|_{L^{qs'_0}}.
\end{align*}
Define $q_0=Q^*s'_0$ and $q_k = \theta^k q_0$ for $k \geq 1$. Notice that $C_q=\frac{q^2}{4(q-1)}\leq Cq,$ where $C$ is independent of $q$.
\begin{align*}
     \|u_j^+\|_{L^{q_0 \theta}} \leq (Cq_0)^{\frac{1}{q_0}}\|V\|_{L^{s_0}}^{\frac{1}{q_0}}  \|u_j^+\|_{L^{q_0}}.
\end{align*}
\begin{align*}
    \|u_j^+\|_{L^{q_k}} \leq \left({\displaystyle \prod_{i=0}^{k-1}} (Cq_i)^{\frac{1}{q_i}}\right) \|V\|_{L^{s_0}}^{\sum\limits_{i=0}^{k-1} {\frac{1}{q_i}}} \|u_j^+\|_{L^{q_0}}.
\end{align*}
Recall that $q_k \to \infty$ as $k \to \infty$ and
\begin{align*}
    \sum_{i=0}^{\infty} \frac{1}{q_i} = \frac{1}{q_0} \sum_{i=0}^{\infty} \frac{1}{\theta^i}< \infty \quad \text{and} \quad \sum_{i=0}^{\infty} \frac{\log q_i}{q_i} < \infty.
\end{align*}
Hence, $\|u_j^+\|_{L^{\infty}} \leq C$. Using similar argument, we can prove that $\|u_j^-\|_{L^{\infty}} \leq C$. Finally,
\begin{align*}
    \|u_j\|_{L^{\infty}} \leq \|u_j^+\|_{L^{\infty}} + \|u_j^-\|_{L^{\infty}} \leq C, \forall j \geq j_0.
\end{align*}
Since $a_j \rightarrow 0$ as $j \rightarrow \infty$, there exists some $a_2 \in (0,a_1)$ such that $\|u_a\|_{L^{\infty}} \leq C$, for all $a \in (0,a_2)$.
The continuity of the solution is established through Morrey's embedding and the regularity results \cite[Theorem 5.15 and Theorem 6.1]{Folland1975}.
\end{proof}

\begin{lemma}\label{lemma 2.8}
    There exists $C_1>0$ such that $\|u_a\| \geq C_1$ for all $a \in (0, a_1)$, where $a_1$ is given in Theorem \ref{Existence and uniform boundedness}.
\end{lemma}
\begin{proof}
Notice that $F_a(s) \geq -a|s|$ for all $s \in \mathbb{R}$. Notice that $u_a$ is a solution to \eqref{Main problem} for all $a \in (0,a_1)$. Therefore, using the mountain pass geometry and $\delta>0$ as given in Lemma \ref{MP1}, we write
\begin{align}\label{delta estimate}
     J_a(u_a) \geq \delta \text{ for all } a \in (0, a_1).
\end{align}
Further, we have the following estimate
\begin{align*}
    \delta &\leq J_a(u_a) = \frac{1}{2} \|u_a\|^2 - \int\limits_{\mathbb{H}^N} g(\xi) F_a(u_a) \dxi \leq  \frac{1}{2} \|u_a\|^2 + a\int\limits_{\mathbb{H}^N} g(\xi) |u_a| \dxi \leq  \frac{1}{2} \|u_a\|^2 + a C \|u_a\|.
\end{align*}
The above inequality further implies
\begin{align*}
    \|u_a\| \geq \frac{\delta}{\frac{1}{2}\|u_a\| + a_1C} \geq \frac{\delta}{\frac{1}{2}\widetilde{C} + a_1C} := C_1.
\end{align*}
The last inequality follows using the uniform boundedness of the solutions in $\Dh$.
\end{proof}
\begin{lemma} \label{linfinity lower bound}
There exists $\beta>0$ such that $\|u_a\|_{L^{\infty}} \geq \beta$ for all $a \in (0, a_2)$, where $\beta$ is independent of $a$.
\end{lemma}
\begin{proof}
From \eqref{delta estimate} and the lower estimate on $F_a$ as given in Lemma \ref{lemma 2.8}, for all $a\in (0,a_1)$ we have
\begin{align*}
    \frac{1}{2} \|u_a\|^2 &= J_a(u_a) + \int\limits_{\mathbb{H}^N} g(\xi) F_a(u_a) \dxi \geq \delta - a\int\limits_{\mathbb{H}^N} g(\xi)|u_a| \dxi \geq \delta - aC \|u_a\| \geq \delta - aCC_1,
\end{align*}
where $C$ and $C_1$ are positive constants. For simplicity, we denote $C_2=CC_1$. Choosing $a_2$ such that $0<a_2<\min\{\frac{\delta}{C_2}, a_1\}$ results into the following estimate
\begin{align*}
   \frac{\|u_a\|^2}{2} \geq \delta_0 = \delta -a_2C_2 >0, \text{ for } a \in (0,a_2). 
\end{align*}
Since $u_a$ is a weak solution to \eqref{Main problem}, the following identity holds for every $\varphi \in \Dh$
\begin{align*}
    \int\limits_{\H^N} \nabla_\H u_a \cdot \nabla_\H \varphi\, \dxi = \int\limits_{\H^N} g(\xi)f_a(u_a) \varphi(\xi) d\xi.
\end{align*}
Substituting $\varphi =u_a$ in the above identity, we get
\begin{align*}
       \int\limits_{\H^N} |\nabla_\H u_a |^2 \dxi = \int\limits_{\H^N} g(\xi)f_a(u_a) u_a d\xi.
\end{align*}
For every $a \in (0, a_2)$, we get $\int\limits_{\H^N} g(\xi)f_a(u_a) u_a d\xi \geq 2\delta_0 > 0$. Further, the estimate \eqref{2.3} gives
\begin{align*}
    2\delta_0 &\leq \int\limits_{\H^N} g(\xi)\left(C(|u_a|^{\gamma}+ |u_a|) \right) \dxi \leq \widetilde{C}\left(\|u_a\|_{L^{\infty}}^{\gamma}+ \|u_a\|_{L^{\infty}}\right).
\end{align*} 
Hence, there exists some $\beta>0$ such that $\|u_a\|_{L^{\infty}} \geq \beta$, for every $a \in (0,a_2)$.
\end{proof}

As in Theorem \ref{Regularity-of-solutions}, we consider $a_j \to 0$ as $j \to \infty.$ Let $\{u_j\} = \{u_{a_j}\}$ be the solutions to \eqref{Main problem} and $u_j \rightharpoonup u$ in $\Dh$ as $j \to \infty.$ Now consider the function $f_0$
\begin{align*}
   f_0(s) = 
\begin{cases}
    f(s) & \text{ if } s \geq 0 \\
    0 & \text{ if } s < 0. \\
\end{cases} 
\end{align*}

\begin{proposition}
	\label{prop:1}
Let $\{u_j\}=\{u_{a_j}\}$ be a sequence of solutions to \eqref{Main problem} such that $u_j \rightharpoonup u$ in $\Dh$. Then $u$ solves weakly the positone problem
\begin{equation}\label{Positone problem}
	-\Delta_\mathbb{H} u = g(\xi) f_0(u)\quad \hbox{in $\mathbb{H}^N$}. 
\end{equation}
Moreover, $u\in L^p(\mathbb{H}^N)\cap C(\mathbb{H}^N)$ for every $p\in [Q^*,\infty]$.
\end{proposition}
\begin{proof}
Since $a_j \rightarrow 0$ as $j \rightarrow \infty$, there exists some $j_1\in\mathbb{N}$ such that $a_j \leq a_2$ for all $j \geq j_1$. Next, we prove that
\begin{align*}
   \int\limits_{\mathbb{H}^N} g(\xi) f_j (u_j) \varphi(\xi) d\xi \to \int\limits_{\mathbb{H}^N} g(\xi) f_0(u)\varphi(\xi) d\xi,\text{ as } j \to \infty, 
\end{align*}
for every $\varphi\in \Dh$ with $\varphi \geq 0$. The triangle inequality gives
\begin{align*}
     |f_j(u_j) - f_0(u)| \leq |f_j(u_j) - f_0 (u_j)| + |f_0(u_j) - f_0(u)|.
\end{align*}
The continuity of $f_0$ gives $f_0(u_j) \to f_0(u)$ as $j \to \infty$. Hence, the dominated convergence theorem yields
\begin{align*}
   \int\limits_{\mathbb{H}^N} g(\xi) |f_0(u_j) - f_0(u)| \varphi(\xi) d\xi  \to 0, \text{ as } j\to \infty. 
\end{align*}
Notice that $\abs{f_j(u_j) - f_0(u_j)} \leq a_j$. We can deduce that
\begin{align*}
    \int\limits_{\mathbb{H}^N} g(\xi) \abs{f_j(u_j) - f_0(u)} \varphi(z) d\xi \leq a_j \int\limits_{\mathbb{H}^N} g(\xi)\varphi(\xi) \dxi + \int\limits_{\mathbb{H}^N} g(\xi) |f_0(u_j) - f_0(u)| \varphi(\xi) d\xi.
\end{align*}
The right hand side in the above inequality tends to zero as $j \rightarrow \infty$. From \eqref{weak formulation}, we have
\begin{align*}
     \int\limits_{\mathbb{H}^N} \nabla_\mathbb{H} u_j (\xi) \nabla_\mathbb{H} \varphi(\xi) d\xi =  \int\limits_{\mathbb{H}^N} g(\xi) f_j(u_j) \varphi(\xi) d\xi, \forall \varphi\in \Dh, \quad\varphi\geq 0.
\end{align*}
Passing through the limit as $j\to\infty$, we get
\begin{align*}
     \int\limits_{\mathbb{H}^N} \nabla_\mathbb{H} u(\xi) \nabla_\mathbb{H} \varphi(\xi) d\xi = \int\limits_{\mathbb{H}^N} g(\xi) f_0(u) \varphi(\xi) \dxi, \forall \varphi \in \Dh, \quad\varphi\geq 0.
\end{align*}
For any $\varphi\in \Dh$, we can write $\varphi = \varphi^+ - \varphi^-$, where $\varphi^+ = \max\{\varphi,0\}$ and $\varphi^- = \max\{-\varphi,0\}$. Since $\varphi^+, \varphi^- \in \Dh$ and both are non-negative. We deduce the following identity:
\begin{equation}\label{2.24}
\int\limits_{\mathbb{H}^N} \nabla_\mathbb{H} u(\xi) \nabla_\mathbb{H} \varphi(\xi) d\xi = \int\limits_{\mathbb{H}^N} g(\xi) f_0(u) \varphi(\xi) d\xi,\quad \forall\varphi\in \Dh.
\end{equation}
The identity \eqref{2.24} infers that $u$ is a weak solution to \eqref{Main problem}. Thus, we get $u \in L^p(\mathbb{H}^N), p \in [Q^*,\infty]$. The continuity of $u$ also follows from Theorem \ref{Regularity-of-solutions}.
\end{proof} 

\begin{remark}
Every weak solution to \eqref{Positone problem} is non-negative. In particular, if $u$ is a weak solution to \eqref{Positone problem}, then choosing the test function $\phi = u^-$ into \eqref{2.24} yields 
\begin{equation}\label{3.11}
    \int\limits_{\mathbb{H}^N} \nabla_\mathbb{H} u(\xi) \nabla_\mathbb{H} u^-(\xi) d\xi \geq 0.
\end{equation}
Recall that
\begin{align}
\begin{split} \label{2.25}
    \int\limits_{\mathbb{H}^N} \nabla_\mathbb{H} u(\xi) \nabla_\mathbb{H} u^-(\xi) d\xi = - \int\limits_{\mathbb{H}^N} |\nabla_\mathbb{H} u^-(\xi)|^2 d\xi \leq 0.
\end{split}
\end{align}
From \eqref{3.11} and \eqref{2.25}, we get $\|u^-\|=0$, which implies $u^- \equiv 0$ a.e in $\mathbb{H}^N$. 
\end{remark}

In the following, We state the maximum principle given in \cite{Wang2001} that we will use in the further section.
\begin{proposition} \label{proposition 3.7}
    Let $ u \in \Dh$ satisfy $- \Delta_{\H} u \geq 0 $, i.e.,
$$
 \int\limits_{\H^N} \nabla_{\H} u \cdot \nabla_{\H} v \geq 0 \quad \text{for all } v \in \Dh \text{ with } v \geq 0.
$$
Then either $ u \equiv 0 $ or $ u $ cannot attain its minimum in $\H^N$.
\end{proposition}
Applying Proposition \ref{proposition 3.7} to \eqref{Positone problem} and using Lemma \ref{linfinity lower bound}, it follows that $u>0$ a.e. in $\H^N$.
\section{Towards the positivity of the solution}\label{section4}
\begin{proposition} \label{uniform convergence}
 Let $\{u_{a_j}\}$ be a sequence of solutions to \eqref{Main problem} such that $u_{a_j} \rightharpoonup u$ in $\Dh$ as $j \rightarrow \infty$. Then, $\norm{u_{a_j} -u}_{L^\infty} \rightarrow 0$ as $j \rightarrow \infty$. In particular, there exists $a_3 \in (0,a_2)$ such that $u_a \geq 0$ for all $a \in (0,a_3)$.
\end{proposition}
\begin{proof}
The fundamental solution of the sub-Laplacian on $\mathbb{H}^N$ is given by $ \mathcal{F}(\xi) =   C_N \frac{1}{|\xi|_{\mathbb{H}}^{Q-2}}, \xi \neq 0$. Recall that $u$ satisfies $-\Delta_{\mathbb{H}} u = g(\xi) f_0(u) \text{ in } {\mathbb{H}^N}$. Thus, the Riesz potential formula \cite[Corollary 1]{Folland1973} yields
\begin{align*}
   u(\xi) =  C_N\int\limits_{\mathbb{H}^N} \frac{g(\eta)f_0(u(\eta))}{|\eta^{-1} \circ \xi|_{\mathbb{H}}^{Q-2}} \deta \quad\text{ and }\quad   u_j(\xi)  =  C_N\int\limits_{\mathbb{H}^N} \frac{g(\eta)f_j(u(\eta))}{|\eta^{-1} \circ \xi|_{\mathbb{H}}^{Q-2}}\deta.
\end{align*}
Now we get the following estimate
\begin{align*}
    |u(\xi) - u_j(\xi)| &\leq  C_N\int\limits_{\mathbb{H}^N} g(\eta) \frac{|f_j(u(\eta)) - f_0(u(\eta))|}{|\eta^{-1} \circ \xi|_{\mathbb{H}}^{Q-2}} \deta \\
    &\leq C_N\int\limits_{B_1} g(\eta) \frac{|f_j(u(\eta)) - f_0(u(\eta))|}{|\eta^{-1} \circ \xi|_{\mathbb{H}}^{Q-2}} \deta + C_N\int\limits_{\mathbb{H}^N \setminus B_1} g(\eta) \frac{|f_j(u(\eta)) - f_0(u(\eta))|}{|\eta^{-1} \circ \xi|_{\mathbb{H}}^{Q-2}} \deta,
\end{align*}
where $B_1=\{\eta \in \H^N : |\eta^{-1} \circ \xi|_{\mathbb{H}} < 1, \xi \in \H^N \}$. We choose $\theta >1$ such that $\frac{2 \theta}{\theta-1}>Q$ and using H\"{o}lder's inequality with conjugate pairs $(\theta,\theta')$, the first integral is estimated as follows:
\small
\begin{align*}
    \int\limits_{B_1} g(\eta) \frac{|f_j(u(\eta)) - f_0(u(\eta))|}{|\eta^{-1} \circ \xi|_{\mathbb{H}}^{Q-2}} \deta &\leq \left(  \int\limits_{B_1}  \frac{1}{|\eta^{-1} \circ \xi|_{\mathbb{H}}^{(Q-2)\theta}} \deta\right)^{\frac{1}{\theta}}\left(\int\limits_{B_1} g(\eta)^{\theta'}|f_j(u(\eta)) - f_0(u(\eta))|^{\theta'} \mathrm{d} \eta \right)^\frac{1}{\theta'}.
\end{align*} 
Using the polar integration formula on $\mathbb{H}^N$\cite[Proposition 1.15]{folland1982hardy}, we get
\begin{align*}
    \int\limits_{B_1}  \frac{1}{|\eta^{-1} \circ \xi|_{\mathbb{H}}^{(Q-2)\theta}}\,\mathrm{d}\eta = \omega_N \int\limits_0^1 \frac{r^{Q-1}}{r^{(Q-2)\theta}}\,\mathrm{d}r = C(N,\theta).
\end{align*}
Proceeding as in Theorem \ref{Regularity-of-solutions}, we get $u_j \to u$ in $L^q(\mathbb{H}^N),$ for $q \in [Q^*,\infty)$. Then, the dominated convergence theorem gives 
\begin{align*}
    \int\limits_{B_1} g(\eta)^{\theta'}|f_j(u(\eta)) - f_0(u(\eta))|^{\theta'} d \eta \to 0, \text{ as } j \to \infty.
\end{align*}
Further, we estimate
\begin{align*}
   \int\limits_{\mathbb{H}^N \setminus B_1} g(\eta) \frac{|f_j(u(\eta)) - f_0(u(\eta))|}{|\eta^{-1} \circ \xi|_{\mathbb{H}}^{Q-2}} \deta \leq \int\limits_{\mathbb{H}^N \setminus B_1} g(\eta) |f_j(u(\eta)) - f_0(u(\eta))| \deta. 
\end{align*}
The dominated convergence theorem gives,
\begin{align*}
    \int\limits_{\mathbb{H}^N \setminus B_1} g(\eta) \frac{|f_j(u(\eta)) - f_0(u(\eta))|}{|\eta^{-1} \circ \xi|_{\mathbb{H}}^{Q-2}} \deta \to 0, \quad \text{as } j \to \infty.
\end{align*}
Hence, we obtain $u_j \to u$ in $L^\infty(\mathbb{H}^N)$.
\end{proof}

\begin{proof}[Proof of the Theorem \ref{thm 1.1}]
(a). The proof of this part follows using Proposition \ref{uniform convergence}.\\

\noindent (b) For $j \in \mathbb{N}$ sufficiently large, since $f_j$ is locally Lipschitz and $0 \leq u_j,u \leq C,$ we have
\begin{align*}
    |f_j(u_j(\xi)) - f_0(u(\xi))|  \leq |f_j(u_j(\xi)) - f_j(u(\xi))| + |f_j(u(\xi)) - f_0(u(\xi))
    \leq L|u_j(\xi) - u(\xi)| + a_j,
\end{align*}
where $L$ is the Lipschitz constant. Further, We have
\begin{align*}
    |u_j(\xi) - u(\xi)| &\leq  C_N\int\limits_{\mathbb{H}^N} g(\eta) \frac{|f_j(u(\eta)) - f_0(u(\eta))|}{|\eta^{-1} \circ \xi|_{\mathbb{H}}^{Q-2}} \deta \\
    &\leq  C_N L \int\limits_{\mathbb{H}^N} g(\eta) \frac{|u_j(\xi) - u(\xi)|}{|\eta^{-1} \circ \xi|_{\mathbb{H}}^{Q-2}} \deta +C_N a_j \int\limits_{\mathbb{H}^N}  \frac{g(\eta)}{|\eta^{-1} \circ \xi|_{\mathbb{H}}^{Q-2}} \deta\\
    & \leq C_N C\left(\|u_j - u\|_{L^\infty(\mathbb{H}^N)} + a_j \right) \frac{1}{|\xi|_{\mathbb{H}}^{Q-2}}.
\end{align*}
Thus, we get
\begin{align*}
    \sup\limits_{\xi \in \mathbb{H}^N \setminus \{0\}} |\xi|_{\mathbb{H}}^{Q-2} |u_j(\xi) - u(\xi)| \to 0 \text{ as } j \to \infty.
\end{align*}
We next show that $ \lim\limits_{|\xi|_{\mathbb{H}} \to \infty} |\xi|_{\mathbb{H}}^{Q-2} u(\xi) > 0$. For that part, we notice that
\begin{align*}
    \lim\limits_{|\xi|_{\mathbb{H}} \to \infty} |\xi|_{\mathbb{H}}^{Q-2} u(\xi) &=  \lim\limits_{|\xi|_{\mathbb{H}} \to \infty} C_N \int\limits_{\mathbb{H}^N} \frac{g(\eta)f_0(u(\eta))|\xi|^{Q-2}}{|\eta^{-1} \circ \xi|_{\mathbb{H}}^{Q-2}} \deta \geq \lim\limits_{|\xi|_{\mathbb{H}} \to \infty} C_N  \int\limits_{B_R} \frac{g(\eta)f_0(u(\eta))|\xi|^{Q-2}}{|\eta^{-1} \circ \xi|_{\mathbb{H}}^{Q-2}} \deta,
\end{align*}
for any $R>0.$ Now for any $R>0,$ there exists $\xi \in \mathbb{H}^N$ such that $|\xi| > 2R+1.$
For all $\eta \in B_R$,
\begin{align*}
    |\eta^{-1} \circ \xi|_{\mathbb{H}}^{Q-2} &\geq ||\xi|_{\mathbb{H}} - |\eta|_{\mathbb{H}}|^{Q-2} \geq ||\xi|_{\mathbb{H}} - R|^{Q-2} \geq \left||\xi|_{\mathbb{H}} - \frac{|\xi|_{\mathbb{H}} - 1}{2}\right|^{Q-2} \\
    & \geq \left|\frac{|\xi|_{\mathbb{H}} + 1}{2}\right|^{Q-2}\\
    &= 2^{2-Q}(|\xi|_{\mathbb{H}} + 1)^{Q-2}.
\end{align*}
Thus, for $\eta \in B_R$, we can write
\begin{align*}
    \frac{|\xi|_{\mathbb{H}}^{Q-2}}{|\eta^{-1} \circ \xi|_{\mathbb{H}}^{Q-2}}g(\eta) f_0(u(\eta)) &\leq 2^{Q-2}g(\eta) f_0(u(\eta)).
\end{align*}
Moreover,
\begin{align*}
     \frac{|\xi|_{\mathbb{H}}^{Q-2}}{|\eta^{-1} \circ \xi|_{\mathbb{H}}^{Q-2}}g(\eta) f_0(u(\eta)) \to  g(\eta) f_0(u(\eta)), \text{ in } B_R  \text{ as } |\xi| \to \infty.
\end{align*}
The dominated convergence theorem yields
\begin{align*}
    \lim\limits_{|\xi|_{\mathbb{H}} \to \infty}  \int\limits_{B_R} \frac{g(\eta) f_0(u(\eta))|\xi|_{\mathbb{H}}^{Q-2}}{|\eta^{-1} \circ \xi|_{\mathbb{H}}^{Q-2}}g(\eta) f_0(u(\eta)) \deta =   \int\limits_{B_R} g(\eta) f_0(u(\eta)) \deta.
\end{align*}
Hence,
\begin{align*}
    \lim\limits_{|\xi|_{\mathbb{H}} \to \infty}  |\xi|^{Q-2} u(\xi)  \geq  C \int\limits_{B_R} g(\xi) f_0(u(\xi)) \dxi.
\end{align*}
Using Fatou's lemma, we get
\begin{align*}
    \lim\limits_{R \to \infty} \lim\limits_{|\xi|_{\mathbb{H}} \to \infty}  |\xi|^{Q-2} u(\xi)  \geq  C \lim\limits_{R \to \infty} \int\limits_{B_R} g(\xi) f_0(u(\xi)) \dxi \geq \int\limits_{\mathbb{H}^N} g(\xi) f_0(u(\xi)) \dxi,
\end{align*}
Thus, $ \lim\limits_{|\xi|_{\mathbb{H}} \to \infty}  |\xi|^{Q-2} u(\xi) > 0$. Consequently, there exists $j_2 \in \mathbb{N}$ and $R>>1$ such that for $j \geq j_2, u_j>0 $ a.e. on $\H^N \setminus B_R$. Also since $u$ is continuous, there exists $\gamma > 0$ such that $u > \gamma$ on $\overline{B_R}$ since $u>0$. Moreover, by Proposition \ref{uniform convergence}, there exists $j_3 \in \mathbb{N}$ such that for $j \geq j_3, u_j>0$ a.e. on $\overline{B_R}$. Finally, we choose $j_4=\max{\{j_2,j_3\}}$ so that $j \geq j_4, u_j>0$ on $\mathbb{H}^N$. In other words, then there exists $a_4 \in (0,a_3)$ such that $u_a > 0$ for all $a \in (0,a_4)$.
\end{proof} 
\begin{example}
(a) Let us consider a continuous function $f:[0,\infty) \rightarrow [0,\infty)$ as follows:
\begin{align*}
    f(t) = A t^{\gamma-1},
\end{align*}
where $A$ is a positive real constant and $\gamma \in (2,Q^*)$. Then we have
\begin{align*}
    \lim\limits_{t \rightarrow 0^+} \frac{f(t)}{t} = A \lim\limits_{t \rightarrow 0^+} t^{\gamma-2}= 0,
\end{align*}
and 
\begin{align*}
    \lim\limits_{t \rightarrow + \infty} \frac{f(t)}{t^{\gamma-1}} = \lim\limits_{t \rightarrow + \infty} \frac{A t^{\gamma-1}}{t^{\gamma-1}} =A.
\end{align*}
Thus the assumption \ref{f1} is verified using the above two limits.  Further, we have
\begin{align*}
    F(t) = \int_{0}^{t} f(\tau)\, \mathrm{d}\tau = \frac{A t^\gamma}{\gamma},
\end{align*}
which further implies $\gamma F(t) = t f(t)$. Hence, the assumption \ref{f2} is also verified.\\[5mm]
(b) Let us consider the function $g$ as follows:
\begin{align*}
    g(\xi) = \begin{cases}
        \frac{1}{|\xi|^{\alpha}} &\text{ \rm for } \xi \in \H^N \setminus B_1(0) \text{ and for some } \alpha > Q,\\
        1 &\text{\rm for } \xi \in B_1(0).
    \end{cases}
\end{align*}
Notice that $|g(\xi)| \leq 1$ for all $\xi \in \H^N$ and consequently we get $g \in L^\infty(\H^N)$. On the other hand, we have
\begin{align*}
    \int_{\H^N} g(\xi)\,\mathrm{d}\xi = \int_{B_1(0)} g(\xi)\,\mathrm{d}\xi + \int_{\H^N \setminus B_1(0)} g(\xi)\,\mathrm{d}\xi.
\end{align*}
The first integral on the right hand side is finite since $g$ is constant on $B_1(0)$. Using the polar integration formula, we evaluate the second integral as follows:
 \begin{align*}
     \int\limits_{\H^N \setminus B_1(0)} g(\xi) \dxi = \int\limits_{\H^N \setminus B_1(0)} \frac{1}{|\xi|^\alpha} \dxi 
     = \omega_N \int\limits_1^\infty \frac{r^{Q-1}}{r^\alpha} \dr
     =  \frac{\omega_N}{\alpha - Q},
 \end{align*}
 which is finite for $\alpha > Q.$ Hence, we deduce that $g \in L^1(\H^N)$. Therefore, the assumption \ref{g1} on $g$ is verified.
\end{example}

\section*{Aknowledgement}
R. Kumar acknowledges the financial support from the 'Northeast Petroleum University (Daqing, China) Postdoctoral International Exchange Program'.
The authors are also thankful to Dr. Gaurav Dwivedi (BITS Pilani) for carefully reading the manuscript. We sincerely thank the anonymous reviewers for the careful review and insightful corrections which improved the manuscript.


\begin{thebibliography}{10}

\bibitem{Alves2020}
C.~O. Alves, A.~R.~F. de~Holanda, and J.~A. dos Santos.
\newblock Existence of positive solutions for a class of semipositone problem in whole {$\mathbb{R}^N$}.
\newblock {\em Proc. Roy. Soc. Edinburgh Sect. A}, 150(5):2349--2367, 2020.

\bibitem{Alves2019}
C.~O. Alves, A.~R.~F. de~Holanda, and J.~A. Santos.
\newblock Existence of positive solutions for a class of semipositone quasilinear problems through {O}rlicz-{S}obolev space.
\newblock {\em Proc. Amer. Math. Soc.}, 147(1):285--299, 2019.

\bibitem{Ambrosetti1994}
A.~Ambrosetti, D.~Arcoya, and B.~Buffoni.
\newblock Positive solutions for some semi-positone problems via bifurcation theory.
\newblock {\em Differential Integral Equations}, 7(3-4):655--663, 1994.

\bibitem{Ambrosetti1973}
A.~Ambrosetti and P.~H. Rabinowitz.
\newblock Dual variational methods in critical point theory and applications.
\newblock {\em J. Functional Analysis}, 14:349--381, 1973.

\bibitem{Bai2024}
S.~Bai, D.~D.~Repov{\v{s}}, and Y.~Song,
\newblock High and low perturbations of the critical Choquard equation on the Heisenberg group,
\newblock {\em Adv. Differ. Equ.} 29(3-4):153--178, 2024.

\bibitem{Biswas2024}
N.~Biswas.
\newblock Study of fractional semipositone problems on {$\mathbb R^N$}.
\newblock {\em Opuscula Math.}, 44(4):445--470, 2024.

\bibitem{Biswas2023}
N.~Biswas, U.~Das, and A.~Sarkar.
\newblock On the fourth order semipositone problem in {$ \mathbb{R^N}$}.
\newblock {\em Discrete Contin. Dyn. Syst.}, 43(1):411--434, 2023.

\bibitem{biswassemipositone}
N.~Biswas and R.~Kumar.
\newblock On semipositone problems over {$\mathbb{R}^N$} for the fractional {$p$}-{L}aplace operator.
\newblock {\em J. Math. Anal. Appl.}, 551(2):Paper No. 129703, 32, 2025.

\bibitem{Bonfiglioli2007}
A.~Bonfiglioli, E.~Lanconelli, and F.~Uguzzoni.
\newblock {\em Stratified {L}ie groups and potential theory for their sub-{L}aplacians}.
\newblock Springer Monographs in Mathematics. Springer, Berlin, 2007.

\bibitem{Caldwell2007}
S.~Caldwell, A.~Castro, R.~Shivaji, and S.~Unsurangsie.
\newblock Positive solutions for classes of multiparameter elliptic semipositone problems.
\newblock {\em Electron. J. Differential Equations}, pages No. 96, 10, 2007.

\bibitem{Castro2016}
A.~Castro, D.~G. de~Figueredo, and E.~Lopera.
\newblock Existence of positive solutions for a semipositone {$p$}-{L}aplacian problem.
\newblock {\em Proc. Roy. Soc. Edinburgh Sect. A}, 146(3):475--482, 2016.

\bibitem{Castro1995}
A.~Castro, M.~Hassanpour, and R.~Shivaji.
\newblock Uniqueness of non-negative solutions for a semipositone problem with concave nonlinearity.
\newblock {\em Comm. Partial Differential Equations}, 20(11-12):1927--1936, 1995.

\bibitem{Castro1988}
A.~Castro and R.~Shivaji.
\newblock Nonnegative solutions for a class of nonpositone problems.
\newblock {\em Proc. Roy. Soc. Edinburgh Sect. A}, 108(3-4):291--302, 1988.

\bibitem{Castro1989}
A.~Castro and R.~Shivaji.
\newblock Nonnegative solutions for a class of radially symmetric nonpositone problems.
\newblock {\em Proc. Amer. Math. Soc.}, 106(3):735--740, 1989.

\bibitem{Chhetri2015}
M.~Chhetri, P.~Dr\'{a}bek, and R.~Shivaji.
\newblock Existence of positive solutions for a class of {$p$}-{L}aplacian superlinear semipositone problems.
\newblock {\em Proc. Roy. Soc. Edinburgh Sect. A}, 145(5):925--936, 2015.

\bibitem{Costa2006}
D.~G. Costa, H.~Tehrani, and J.~Yang.
\newblock On a variational approach to existence and multiplicity results for semipositone problems.
\newblock {\em Electron. J. Differential Equations}, pages No. 11, 10, 2006.


\bibitem{Dhanya2021}
R.~Dhanya and S.~Tiwari.
\newblock A multiparameter fractional {L}aplace problem with semipositone nonlinearity.
\newblock {\em Commun. Pure Appl. Anal.}, 20(12):4043--4061, 2021.

\bibitem{Folland1973}
G.~B.~Folland,
\newblock A fundamental solution for a subelliptic operator,
\newblock {\em Bull. Amer. Math. Soc.} 79:373--376, 1973.

\bibitem{Folland1975}
G.~B. Folland.
\newblock Subelliptic estimates and function spaces on Nilpotent {L}ie groups.
\newblock {\em Ark. Mat.}, 13(2):161--207, 1975.

\bibitem{folland1982hardy}
G.~B.~Folland and E.~M.~Stein.
\newblock {\em Hardy spaces on homogeneous groups}.
\newblock Princeton University Press, 1982.

\bibitem{Liang2024}
S.~Liang, P.~Pucci, Y.~Song, and X.~Sun.
\newblock On a critical {C}hoquard-{K}irchhoff {$p$}-sub-{L}aplacian equation in {$\mathbb{H}^n$}.
\newblock {\em Anal. Geom. Metr. Spaces}, 12, pages No. 28, 2024.

\bibitem{Lopera2023}
E.~Lopera, C.~L\'opez, and R.~E. Vidal.
\newblock Existence of positive solutions for a parameter fractional {$p$}-{L}aplacian problem with semipositone nonlinearity.
\newblock {\em J. Math. Anal. Appl.}, 526(2):Paper No. 127350, 12, 2023.

\bibitem{molica2024mountain}
G.~Molica~Bisci, P.~Malanchini, and S.~Secchi.
\newblock Mountain pass solutions for an entire semipositone problem involving the Grushin subelliptic operator.
\newblock {\em Topological Methods in Nonlinear Analysis}, Vol. 67, no. 1, pp. 183 - 208, 2026.

\bibitem{Sun2022}
X.~Sun, Y.~Song, and S.~Liang.
\newblock On the critical {C}hoquard-{K}irchhoff problem on the {H}eisenberg group.
\newblock {\em Adv. Nonlinear Anal.}, 12(1):210--236, 2023.

\bibitem{Wang2001}
W.~Wang.
\newblock Positive solution of a subelliptic nonlinear equation on the {H}eisenberg group.
\newblock {\em Canad. Math. Bull.}, 44(3):346--354, 2001.

\end{thebibliography}
\end{document}